\documentclass[11pt,leqno]{amsart}
\usepackage{amsfonts, amsmath, amsthm, amssymb,xspace}
\usepackage[breaklinks=true]{hyperref}
\usepackage{amsmath,amscd}
\theoremstyle{plain}
\newtheorem{theorem}{Theorem}[section]
\newtheorem{lemma}{Lemma}[section]

\theoremstyle{definition}
\newtheorem{defin}{Definition}[section]

\newtheorem*{acknowledgement}{Acknowledgements}

\newcommand{\mf}[1]{\displaystyle{\mathfrak{#1}}}

\newcommand{\comment}[1]{}

\DeclareMathOperator{\spec}{\ensuremath{Spec}}
\DeclareMathOperator{\Gr}{\ensuremath{gr}}

\DeclareMathOperator{\Sym}{\ensuremath{Sym}}

\begin{document}

\title{On the Azumaya locus of almost commutative algebras }
\author{Akaki Tikaradze}
\address{The University of Toledo, Department of Mathematics, Toledo, Ohio, USA}
\email{\tt atikara@utnet.utoledo.edu}
\date{\today}

\begin{abstract} We prove a general statement which implies  the 
coincidence of the Azumaya and smooth loci of the center of an algebra in 
positive characteristic,
provided that the spectrum of its associated graded algebra has a large symplectic
leaf. In particular, we show that for a symplectic
reflection algebra smooth and the Azumaya loci coincide.

\end{abstract}
\maketitle
\section{Introduction}

Throughout, we will fix a ground field $\bold{k},$ which will be assumed to
be algebraically closed with positive characteristic $p$. In studying 
the representation
theory of an associative $\bold{k}$-algebra $A$ with a large center (i.e. $A$
is a finitely generated module over its center), it is important
to understand the Azumaya locus of $A.$ Recall that the Azumaya
locus of $A$ is defined to be the subset of $\spec \bold{Z}(A)$ (where $\bold{Z}(A)$ 
denotes the center of $A$) consisting of all prime ideals $I\in \spec \bold{Z}(A)$
such that the localized algebra $A_I$ is an Azumaya algebra. If 
$A$ is a prime Noetherian ring with its center $\bold{Z}(A)$ being finitely generated
over $\bold{k}$, then a character $\chi:\bold{Z}(A)\to k$ belongs to the Azumaya
locus if and only if $A_{\chi}=A\otimes_{\bold{Z}(A)}\bold{k}$ affords an
irreducible representation whose
dimension is the largest possible dimension for an irreducible $A$-module,
and this largest possible dimension is the PI-degree of $A$ [\cite{BG} Proposition 3.1].
In particular, the square of PI-degree of $A$ is equal to the rank of $A$ over $\bold{Z}(A).$
 It is well-known that if the algebra $A$ is smooth, then the Azumaya locus
 is contained in the smooth locus of $\spec \bold{Z}(A)$ \cite{BG}. Thus, the natural
 question is when are these two open subsets of $\spec \bold{Z}(A)$ equal.
 In this direction, there is a general result (which will be crucial for us)
 due to Brown and Goodearl [\cite{BG} Theorem 3.8.], which states that if $A$ is a prime Noetherian
 ring which is Auslander-regular and Cohen-Macaulay, such that the complement
 of the Azumaya locus in $\spec \bold{Z}(A)$ has codimension at least 2, then
 the Azumaya and smooth loci coincide. Using this result, Brown and Goodearl \cite{BG} showed
 that the Azumaya locus coincides with the smooth locus when $A$ is either the universal
 enveloping algebra of a reductive Lie algebra in a very good characteristic,
 quantized algebra of functions, quantized enveloping algebras at roots of unity.
 Also, Brown and Changtong \cite{BC} have proved the similar result for rational
 Cherednik algebras \cite{BC}, and we have shown an analogous result for infinitesimal
 Hecke algebras of $\mf{sl}_2$ \cite{T}. 
 
  Given examples above, the natural question is what do the above
  examples have in common, and whether there is a general condition
  which will imply that the Azumaya and smooth loci coincide. 
  In this paper we give such a statement.  
  Namely, we
  will show that if the algebra $A$ can be equipped with a positive
  filtration such that the associated graded algebra has an open
  symplectic leaf whose complement  has codimension greater or equal 
  to 2, then the
  Azumaya locus of $A$ has the complement of codimension at least
  2. The combination of this with the above mentioned result of 
  Brown-Goodearl enables us to state a general result about the 
  coincidence of the Azumaya and smooth loci, in particular we will show
  that for a symplectic reflection algebra the Azumaya locus
  coincides with the smooth locus.
\section{General results} 
 
Let $A$ be an associative algebra over $\bold{k}$ equipped with a nonnegative filtration
$A_n\subset A_{n+1}\subset\cdots$, $n\in \mathbb{Z}_{+}$ such that the associated graded algebra $\Gr A$
is  commutative. Let $d$ be the largest number such that $[A_n, A_m]\subset A_{n+m-d}$
for all $n, m.$ Then $\Gr A$ becomes equipped with the following natural Poisson bracket:
for any homogeneous elements $a\in \Gr A_n, b\in \Gr A_m,$ one defines their Poisson bracket
$\{ a, b\}$ to be the symbol of $[a', b'] \in \Gr A_{n+m-d},$ where $a'\in A_n, b'\in A_m$ are arbitrary lifts of $a, b.$
Similarly, If $A$ is a flat associative algebra over $\bold{k}[t]$, 
such that $A/tA$ is commutative, then commutator bracket of $A$ induces a Poisson bracket on $A/tA$ in the standard way:
If d is the largest integer such that $A/tA$ is commutative, then for any $a, b\in A/tA,$ one
puts $\{ a, b\}=\frac{1}{t^d}[a', b'],$ where $a', b'$ are arbitrary lifts of $a, b$ in $A.$

 We start by recalling some terminology and a result
 by Bezrukavnikov and Kaledin \cite{BK}, which will play
 a crucial role in our proof.
 
 \begin{defin} \cite{BK} A central quantization of a (commutative) Poisson algebra $B$ is a flat associative
 $\bold{k}[[t]]$-algebra $B'$, such that $B=B'/tB'$ and Poisson bracket of B is indeced from the commutator  bracket of B, and the quotient map from the center of $B'$ to the Poisson center
 of $B$ is surjective.
 
 \end{defin}
 Here is the result that we are going to use.
 \begin{theorem}[\cite{BK} Proposition 1.24]\label{bez} If under the assumptions of the above definition, $\spec B$
 is a symplectic variety, then $B'[t^{-1}]$ is a simple algebra over its center.
 \end{theorem}

 The following result will be crucial.

\begin{theorem}\label{muh} Let $A$ be an associative algebra over $\bold{k}$ equipped with an 
increasing algebra filtration
$\bold{k}=A_0\subset A_1\subset\cdots,$ such that $\Gr A$ is a commutative 
finitely generated domain, and
the smooth locus of $X=\spec \Gr A, U\subset X$ 
is a symplectic variety with respect to the Poisson bracket of $\Gr A$ and 
the codimension of  the singular locus $Z=X-U$ in $X$ is $\geq 2.$
If $\Gr \bold{Z}(A)=(\Gr A)^p$, then the complement of the Azumaya locus 
of $\bold{Z}(A)$ inside the smooth locus of $\bold{Z}(A)$ has codimension $\geq 2$. 

\end{theorem}

\begin{proof}

 Let $f\in (\Gr A)^p$ be a nonzero homogeneous element which vanishes on the singular locus
 of $\spec \Gr A,$ so $f(Z)=0$. Thus, by the assumption $\spec (\Gr A_f)$ is a symplectic variety. 
 Let us consider
 an element $g\in \bold{Z}(A)$ such that $\sigma (g)=f$ (from now on $\sigma$ will denote the principal part of an element with respect to the filtration). Let us put
 $S=A_g, \deg(g)=d, d$ is a positive integer. 
 Let us consider
 an induced filtration on $A_g$ coming from the filtration on $A$, 
 namely $\deg g^{-1}=-d$.
 Then, $\Gr S= (\Gr A)_f$. Let us consider the Rees algebra of 
 $S: R(S)=\sum S_mt^m\subset S[t, t^{-1}],$ where $S_m$ denotes the 
 set of elements of $S$ of the filtration degree $\leq m$.
  Clearly $R(A)$ is a finitely generated module
 over its center $\bold{Z}(R(A))=R(\bold{Z}(A))$ (since $R(A)$ is positively graded), and
 since $R(S)$ is the localization $R(A)$ by a central element $gt^{d}$, we see 
 that $R(S)$ is finite over its center $\bold{Z}(R(S))$.
 
 Let us complete $R(S)$ with respect to $t\in R(S).$
Denote this completion by $\overline {R(S)},$ so $\overline {R(S)}=\varprojlim R(S)/t^nR(S)$ . 
We have that $\overline {R(S)}$ is a flat 
module over $k[[t]]$, $\overline {R(S)}/t\overline {R(S)}=(\Gr A)_f.$ Thus by Theorem \ref{bez},
$\overline {R(S)}[t^{-1}]$ is simple over $\overline{R(\bold{Z}(S))}[t^{-1}].$ 
Notice that $\overline {R(S)}=R(S)\otimes_{\bold{Z}(R(S))} \overline{R(\bold{Z}(S))}$, so
$$\overline{R(S)}[t^{-1}]=R(S)[t^{-1}]\otimes_{\bold{Z}(R(S))[t^{-1}]}\overline{R(\bold{Z}(S))}[t^{-1}].$$
But, $R(S)[t^{-1}]=S[t,t^{-1}]$, so we see that 
$\overline{R(S)}[t^{-1}]=S\otimes_{\bold{Z}(S)}\overline{R(\bold{Z}(S))}[t^{-1}],$ 
where we use the embedding $i:\bold{Z}(S)\to \overline{R(\bold{Z}(S))}[t^{-1}]$.
Thus, we see that if 
$I\in \spec \bold{Z}(A)$ is a prime ideal of height 1 in the smooth locus of $\bold{Z}(A)$
such that  $g\notin I$ and $I $ belongs to the image of 
$i^{*}:\spec \overline{R(\bold{Z}(S))}[t^{-1}]\to \spec \bold{Z}(A)$, then there is 
a faithfully flat base change $\bold{Z}(A)_{I}\to B,$ where $B=\overline{R(\bold{Z}(S))}[t^{-1}]_J, J\cap \bold{Z}(A)=I,$ 
such that 
$A_I\otimes_{\bold{Z}(A)_I}B$ is a simple algebra over $B.$ 
Therefore, $A_I$ is simple over $\bold{Z}(A)_I.$ Since $\bold{Z}(A)_I$ is
a regular local ring of dimension 1 and $A_I$ is torsion-free over $\bold{Z}(A)_I$, we get
that $A_I$ is projective over $\bold{Z}(A)_I$. Therefore, $A_I$ is Azumaya, so $I$ belongs to the Azumaya locus of $\bold{Z}(A).$ 

Since we want to show that all primes of height 1 from the smooth locus 
of $\bold{Z}(A)$ belong to the Azumaya locus, it is enough to 
show that for any such prime $I\in \spec\bold{Z}(A),$ there exists $f\in (\Gr A)^p$
such that $f(Z)=0,$ and $I\in i^{*}\spec \overline{R(\bold{Z}(S))}[t^{-1}],$ for some
$g\in A$ with $\sigma(g)=f.$ Notice that  $\overline{R(\bold{Z}(S))}[t^{-1}]$ is a subring
of $\overline{\bold{Z}(S)}((t))$, where $\overline{\bold{Z}(S)}$ is a completion of the
filtered ring $\bold{Z}(S)=\bold{Z}(A)_g$ with respect to negative degree subspaces, 
meaning that $\overline{\bold{Z}(S)}$ is the inverse limit $\bold{Z}(S)/\bold{Z}(S)_n$ as 
$n\to -\infty,$ where $\bold{Z}(S)_n$ denoted the $n$-th degree filtration subspace
of $\bold{Z}(S).$ More precisely, elements of $\overline{R(\bold{Z}(S))}[t^{-1}]$ are
of the form $\sum_i^{\infty} s_it^i,$ where $s_i\in \overline{\bold{Z}(S)}$ such that
limit of $i-\deg(s_i)$ is $\infty$ as $i\to \infty,$ and $s_i=0$ 
for sufficiently small $i<<0.$
Thus, if $j:\bold{Z}(S)\to \overline{\bold{Z}(S)}$ denotes the embedding, then
if $I\in Im(j^{*}), j^{*}:\spec(\overline{\bold{Z}(S)})\to \spec(\bold{Z}(S)),$ then
$I$ belongs to the Azumaya locus. Invertible elements of $\overline{\bold{Z}(S)}$ are
precisely those elements whose principal symbol is a power of $f$ times an element of $\bold{k^{*}}$. 
It follows from 
the following trivial lemma that
a prime ideal $I$ of height 1 is in the image of $j^{*}$ if and only if
$I$ contains no elements whose
principal symbol is a power of $f$ up to a an element of $\bold{k}^{*}.$ 

\begin{lemma} Let $j:R\to S$ be an embedding of integral domains, and let $I$ be
a prime ideal of height 1 of $R$ which belongs to the smooth locus of $\spec R.$ If
$I\cap S^{*}=\emptyset,$ then $I\in im(j^{*}):\spec S\to \spec R.$

\end{lemma}
\begin{proof}Without loss of generality we may assume that $R$ is a local ring with
$I$ being the maximal ideal. Therefore, $I=(g)$ for some $g\in R.$ Since $g\notin S^{*},$
there is a prime ideal $J\in \spec S$ such that $g\in J.$ Therefore, $I=j^{-1}(J).$

\end{proof}

Now what we want follows 
from the following well-known fact.

\begin{lemma} Let $B$ be a nonnegatively filtered finitely generated commutative algebra
over $k$, and let $I\subset B$ be an ideal. Then $ht(I)=ht\Gr I.$
\end{lemma}

\begin{proof} We have that the Gelfand-Kirillov dimension of $B/I$ equals to that
of $\Gr B/\Gr I=\Gr (B/I)$, which implies that $ht(I)=ht(\Gr I),$ since $dim B=dim \Gr B.$

\end{proof}
\end{proof}

Note that if $\spec \Gr A$ consists of finitely many symplectic leaves, and 
$\Gr \bold{Z}(A)=(\Gr A)^p$, then assumptions of Theorem \ref{muh} are satisfied.

 We will also use the following 
 
 \begin{lemma}\label{raccoon} Let $M$ be a finitely generated positively filtered module 
 over a positively filtered
 commutative algebra $H,$ such that $\Gr H$ is a domain and
 $\Gr M$ is finite over $\Gr H.$ Then
 the rank of $M$ over $H$ is equal to the rank of $\Gr M$ over $\Gr H.$
 
 \end{lemma}
\begin{proof}

First of all, it is easy to see that the rank of $\Gr M$ over $\Gr H$ is 
$\leq rank_{H}M.$
Indeed, if $x_1,..., x_n\in H$ are elements such that their principal parts
$\sigma(x_1),\cdots, \sigma(x_n)\in \Gr M$ are $\Gr H$-linearly independent, 
then
$x_1,\cdots, x_n$ are $H$-linearly independent in $M.$ Now let us 
consider $R(M),$ the Rees module of $M$ over the Rees algebra $R(H).$
$R(M)$ is a finitely generated $R(H)$-module. From the fact that 
$R(M)/(t-\lambda)=M, \lambda\neq 0, R(M)/tR(M)=\Gr M$ we see that the generic
dimension of fibers of $R(M)$ over $R(H)$ have dimension equal to $rank_HM$. 
Therefore, by the semi-continuity, we conclude that 
$rank_{\Gr H}\Gr M\geq rank_HM$, so we are done.

\end{proof}
 
 \begin{lemma}\label{sheep}Suppose that an algebra $A$ satisfies all the assumptions
 of Theorem \ref{muh}, except that $\Gr \bold{Z}(A)=(\Gr A)^p$. If $\Gr A$ is normal, then $\Gr \bold{Z}(A)\subset (\Gr A)^p.$

\end{lemma}

\begin{proof} Let $f\in \bold{Z}(A),$ then its top symbol $\bar{f}$ belongs
to the Poisson center of $\Gr A.$ Let $U\subset X$ be the smooth locus of
$X.$ By the assumption $U$ is a symplectic variety and $\bar{f}|U$ belongs to the Poisson center
of $O(U).$ Therefore, $df|_U=0,$ hence, there exists $g\in O(U),$ such that ${\bar{f}|}_U=g^p.$ Since $X-U$
has codimension $\geq 2$ in $X$ and $X$ is normal, $g$ extends to the whole $X.$ Thus, that $\bar{f}\in (\Gr A)^p.$

\end{proof}

\begin{theorem}\label{weiwei}

Let $A$ be an algebra over $\bold{k}$ equipped with an increasing 
algebra filtration $\bold{k}=A_0\subset A_1\subset\cdots\subset A_n\cdots,$
such that $\Gr A$ is a finitely generated smooth commutative domain over $\bold{k}.$
Suppose that there exist a central 
subalgebra of $A,$ $\bold{Z}_0\subset \bold{Z}(A),$ 
such that $\Gr A/(\Gr A)(\Gr Z_0)_{+}$ is a domain whose smooth locus 
is a symplectic variety under the natural Poisson bracket
and whose complement has codimension $\geq 2.$ Assume moreover 
that $\Gr A/(\Gr A)(\Gr \bold{Z}_0)_{+}$ is normal.
If  $(\Gr A)^p\subset \Gr \bold{Z}(A),$ then the smooth and the Azumaya locus
of $A$ coincide, and the PI-degree of $A$ is $p^d,$ where 
$2d=dim(\Gr A)/(\Gr A)(\Gr \bold{Z}_0)_{+}.$

\end{theorem}

\begin{proof}

Let $\chi_0:\bold{Z_0}\to \bold{k}$ be a character, then we can consider the quotient
algebra $A_{\chi_0}=A\otimes_{\bold{Z_0}}\bold{k}.$ Then, $A_{\chi_0}$ comes equipped
with the natural filtration induced from $A$
and $\Gr A_{\chi_{0}}=\Gr A/(\Gr \bold{Z_0})_{+},$ also lemma \ref{sheep} with the assumption 
$\Gr \bold{Z}(A)\subset (\Gr A)^p$ implies that 
$\Gr\bold{Z}(A_{\chi_{0}})=(\Gr A/(\Gr \bold{Z_0})_{+}))^p.$ Thus, we may apply our proposition,
which implies that the Azumaya locus of $\bold{Z}(A_{\chi_0})$ has complement
of codimension at least 2. Now we claim that a character $\chi:\bold{Z}(A)\to \bold{k}$
belongs to the Azumaya locus of $\bold{Z}(A)$ if and only if the corresponding
character $\chi_0:\bold{Z}(A_{\chi_{0}})\to \bold{k}$ belongs to the Azumaya locus
of $A_{\chi_{0}},$ where $\chi_0$ is the restriction of $\chi$ on $\bold{Z_0}.$
It is enough to check that the PI-degree of $A$ is equal to the PI-degree
of $A_{\chi_{0}}$ for any character $\chi_0:\bold{Z_0}\to \bold{k}.$ But this
is clear because the PI-degree of $A_{\chi_0},$ which is 
is equal to the square root of the dimension of the generic fiber of $A_{\chi_{0}}$ (\cite{BG}), 
is greater
or equal to the PI-degree  of $A$ (the square root of the dimension of the generic fiber of $A$) 
by the semi-continuity of the rank of fibers of
$A.$ 
On the other hand, the PI-degree of $A$ is
 the largest dimension of an irreducible
module of $A$, which is greater or equal to the
largest possible dimension of an irreducible
module of $A_{\chi_{0}},$ which is precisely the PI-degree of $A_{\chi}.$
 
 So, let $U\subset \spec \bold{Z}(A)$ be the Azumaya locus of $A$.
We have a map $f:\spec \bold{Z}(A)\to \spec \bold{Z_0}$ corresponding
to the inclusion $\bold{Z_0}\subset \bold{Z}(A).$ Let us denote by $Y$ the complement
of the Azumaya locus $Y=\spec \bold{Z}(A)-U$. Then for any closed point 
$\chi\in \spec \bold{Z_0}$, the intersection $f^{-1}(\chi)\cap Y$ has codimension
at least $2$ in $f^{-1}(\chi).$ So, the codimension of $Y$ in $\spec \bold{Z}(A)$
is at least $ 2$. Now, since $A$ is Auslander-regular and Cohen-Macaulay, the above mentioned
 result of Brown-Goodearl \cite{BG} implies the coincidence of the Azumaya and smooth loci.

 By lemma \ref{raccoon}, the PI degree of $A$ is equal to the rank of $\Gr A/(\Gr A)(\Gr \bold{Z_0})_{+}$
over $(\Gr A/(\Gr A)(\Gr \bold{Z_0})_{+})^p,$ which is $p^d$, where $2d$ is the Krull dimension of 
$\Gr A/(\Gr A)(\Gr \bold{Z_0})_{+}$

\end{proof}

\section{Applications to symplectic reflection algebras and enveloping algebras}
 
 Let us recall the definition of a symplectic reflection algebra. 
Let $V$ be a symplectic $\bold{k}$-vector space with the symplectic
 form $\omega:V\times V\to \bold{k}.$
An element $g\in Sp(V)$ is called a symplectic reflection if $rank(Id-g)=2.$
 To a symplectic reflection $s\in Sp(V)$ one may associate a skew-symmetric
 form $\omega_s:V\times V\to \bold{k}$ which coincides with $\omega$ on $Im(Id-s)$
 and is 0 on $Ker(Id-s).$ Let
 $G\subset Sp(V)$ be a finite group generated by symplectic reflections.
  To a $G$-invariant function $c:S\to \bold{k}$ and $t\in \bold{k},$ where
 $S\subset G$ is the subset of symplectic reflection of $G,$ Etingof and Ginzburg \cite{EG}
 associated an algebra (called  a symplectic reflection algebra) $H_{t,c}$ which
 is defined as a quotient of $\bold{k}[G]\ltimes T(V)$ by the relations
 
  $$[x, y]=t\omega(x, y)+\sum_{s\in S}\omega_{s}(x, y)c(s)s.$$
 
 There is a filtration on $H_{t, s} \deg v=1, v\in V, \deg g=0, g\in G$. The crucial
 property is that $\Gr H_{t, c}=\bold{k}[G]\ltimes \Sym V$ \cite{EG}.
 
 The following theorem answers positively two questions 
  raised by Brown-Changtong [\cite{BC} Questions 6.1, 6.4] , and proved by them in the case of a rational Cherednik algebra.
 
 \begin{theorem}
Let $H_{t, c}$ be a symplectic reflection algebra associated to 
$G\subset Sp(V), dim V=2n, t\neq 0$ and $p$ does not divide the order of $G$, then
the smooth and the Azumaya loci of the center of $H_{t, c}$ coincide, and the PI-degree
of $H_{t, c}$ (the maximal dimension of an irreducible module) is equal to $p^n|G|$ . 
 
 \end{theorem} 
 \begin{proof}

 Let us consider $U_{t, c}=eH_{t, c}e,$ the spherical subalgebra of the symplectic
 reflection algebra $H_{t, c},$ where $e=\frac{1}{|G|}\sum _{g\in G}g$
 is the symmetrizing idempotent of $G$. 
 By a theorem of Etingof \cite{BFG}, $\Gr \bold{Z}(U_{t, c})=((\Sym V)^G)^p=(\Gr U_{t, c})^p.$ 
 Theorem \ref{muh} can be
 applied to $U_{t, c},$ since $\Gr U_{t, c}=S(V)^G\subset Sym V$ has finitely 
 many symplectic
 leaves by a result of Brown-Gordon [\cite{BGo} Proposition 7.4]. So, we get that for all prime
 ideals $I$ of the smooth locus of $\bold{Z}(U_{t, c})$ of height 1, the algebra
 ${U_{t, c}}_{I}$ is Azumaya. But by a result of Brown-Changtong [\cite{BC} Corollary 3.6.], $(H_{t,c})_I$
is Morita equivalent to $(U_{t,c})_{I}$, so the complement of the Azumaya locus in the smooth locus of $H_{t, c}$ has codimension $\geq 2.$ Since $H_{t, c}$ is Auslander-regular
and Cohen-Macaulay (\cite{BC}), by Brown-Goodearl \cite{BG}, smooth
and the Azumaya loci coincide for $\bold{Z}(H_{t, c})$. Lemma \ref{raccoon} applied
to $M=H_{t, c}, H=\bold{Z}(H_{t,c})$ implies that 
the PI-degree is independent of $c$, so we may take $c=0,$ in which case the 
desired statement is clear.

\end{proof}

Applying the considerations from the previous section, we may take $A$ to be the enveloping algebra of a semi-simple Lie
algebra $\mf{g}$, and take $\bold{Z_0}\subset \bold{Z}(A)$ to be the subalgebra obtained
by the symmetrization map applied to the generators of $(\Sym g)^G$, in other
words $\bold{Z_0}$ is the reduction modulo $p$ of the usual characteristic 0 central elements
of $\mf{U}\mf{g}$. Then $\Gr A/(\Gr \bold{Z_0})_{+}=\Sym \mf{g}/(\Sym \mf{g})^G_{+}$
is the ring of coinvariants, which is the ring of functions on the nilpotent
cone of $\mf{g}^{*}$ \cite{Ko}, which as a Poisson variety consists of finitely many
symplectic leaves. Thus, assumptions of Theorem \ref{weiwei} are satisfied, as a result
we obtain that the Azumaya locus of $\mf{U}\mf{g}$ coincides with the smooth
locus of its center, a theorem of Brown-Goodearl \cite{BC}. Note that
we did not use any modular representation theory of $\mf{g}.$

A standard example of an almost commutative algebra for which the Azumaya locus
does not coincide with the smooth locus is the enveloping algebra of the Heisenberg
Lie algebra: A Lie algebra $\mf{g}$ with a basis 
$z, x_1,\cdots, x_n, y_1,\cdots y_n$
and relations $[x_i, y_j]=\delta_{ij}z, [z, x_i]=[z, y_j]=0,$ where $\delta_{ij}$ is the Kronecker symbol.
Then, the center of $\mf{U}\mf{g}$ is the polynomial algebra 
$\bold{k}[z, x_1^p,\cdots,x_n^p,y_1^p,\cdots, y_n^p]$, but the Azumaya locus
is the set of characters which do not vanish on $z.$

  The above theorem can also applied to infinitesimal Hecke algebras of $\mf{sl}_2$ \cite{T}.
We expect many applications of the above result for other infinitesimal Hecke algebras.

\acknowledgement{I would like to thank M. Boyarchenko and especially
V. Ginzburg for useful discussions. I am also grateful to 
the anonymous referee for helpful comments. }

\end{document}